\documentclass[12pt,reqno]{amsart}
\usepackage{amsthm,amssymb,amsmath}
\usepackage[colorlinks=true, linkcolor=blue, urlcolor=blue, citecolor=blue]{hyperref}
\usepackage{xcolor}
\usepackage{mathrsfs}
\usepackage{enumitem}
\usepackage{float}
\usepackage{mathtools}
\usepackage[all,cmtip]{xy}
\usepackage{tikz-cd}

\newtheorem{lemma}{Lemma}[section]
\newtheorem{theorem}[lemma]{Theorem}
\newtheorem{proposition}[lemma]{Proposition}
\newtheorem{corollary}[lemma]{Corollary}

\theoremstyle{definition}
\newtheorem{definition}[lemma]{Definition}
\newtheorem{remark}[lemma]{Remark}

\numberwithin{equation}{section}

\newcommand{\e}{E_{\ast}}
\newcommand{\V}{V_{\ast}}

\begin{document}

\title[Pullback and direct image of parabolic ample and nef bundles]{Pullback and direct image of 
parabolic ample and parabolic nef vector bundles}

\author[A. Bansal]{Ashima Bansal}
\address{Department of Mathematics, Shiv Nadar University, NH91, Tehsil
Dadri, Greater Noida, Uttar Pradesh -- 201314, India}
\email{ashima.bansal@snu.edu.in}

\author[I. Biswas]{Indranil Biswas}
\address{Department of Mathematics, Shiv Nadar University, NH91, Tehsil
Dadri, Greater Noida, Uttar Pradesh -- 201314, India}
\email{indranil.biswas@snu.edu.in, indranil29@gmail.com}

\begin{abstract}
We prove that under a finite surjective map of irreducible smooth complex projective curves, the pullback and direct image of a parabolic ample 
(respectively, parabolic nef) vector bundle is again parabolic ample (respectively, parabolic nef) if and only 
if the original parabolic vector bundle is parabolic ample (respectively, parabolic nef).
\end{abstract}

\maketitle
\section{Introduction}
Let $X$ be an irreducible smooth complex projective curve with a subset of marked distinct $n$ points $D\, =\, 
\{x_1, \cdots , x_n\}\, \subset\, X$ called parabolic points. A parabolic vector bundle $\e$ on $(X,\, D)$ is 
a vector bundle $E$ on $X$ endowed with a weighted filtration of the fiber $E_{x_i}$ over each parabolic point 
$x_i\, \in \,D$, which is called a parabolic structure. The notion of parabolic vector bundles was first introduced 
by C. S. Seshadri in \cite{Se}. In \cite[Section 4]{AB}, a canonical construction of parabolic 
structures on the direct image and pullback of a parabolic vector bundle was defined. Furthermore, the notions 
of parabolic ample and parabolic nef vector bundles were introduced in \cite{Bi2} and \cite{BS}, respectively, 
extending the classical notions of ample and nef vector bundles to the parabolic setting.

In Section 2 of this article, we establish several equivalent conditions characterizing the parabolic
ample vector bundles. Analogous criteria are established for the parabolic nef vector bundles.

It is well known that under a finite surjective morphism of varieties, the pullback of an ample vector bundle 
remains ample, and the converse also valid in the sense that if the pullback is ample then the original vector 
bundle is also ample. Similarly, for a proper surjective morphism, the pullback of a vector bundle is nef if 
and only if the original vector bundle is nef. On the other hand, under a finite flat morphism between smooth 
projective varieties, the direct image of an ample vector bundle need not be ample in general. For example,
for a degree two map $f\,:\, {\mathbb P}^1\, \longrightarrow\, {\mathbb P}^1$, we have
$f_* {\mathcal O}_{{\mathbb P}^1}(1) \,=\, {\mathcal O}_{{\mathbb P}^1}\oplus {\mathcal O}_{{\mathbb P}^1}$.
However in \cite{BLNN}, the authors have proved that the dual of an ample (anti-ample) vector bundle behaves well
under direct image; it was also shown there that the converse of it need not be true even for vector bundles on 
curves.

Motivated by the above facts, we investigate analogous questions for parabolic vector bundles over curves; 
this is done in Section 3 and Section 4. In the parabolic setting, the behavior turns out to be well 
structured: for a non-constant morphism $f\,:\,X\,\longrightarrow\,Y$ between smooth projective curves, a 
parabolic vector bundle $\e$ is parabolic ample (respectively, parabolic anti-ample) if and only if 
$f^{\ast}\e$ is parabolic ample (respectively, parabolic anti-ample). Furthermore, we establish similar 
results for direct images, showing that the direct image of a parabolic ample vector bundle (respectively, 
parabolic anti-ample) is again parabolic ample (respectively, parabolic anti-ample), and the converse also 
holds. The same results hold for parabolic nef vector bundles as well.

\section{Criteria for parabolic ample and Parabolic nef vector bundles}

For standard notation and facts about parabolic vector bundles used here we follow \cite{Bi1} and \cite{MY}. 
Let $X$ be an irreducible smooth complex projective curve. Fix a parabolic divisor $D$ on $X$, which is a 
reduced effective divisor, and let $\e$ be a parabolic vector bundle on $(X,\, D)$. We recall that every 
parabolic vector bundle $\e$ has the Harder-Narasimhan filtration (see \cite{MY}). It is the unique increasing 
filtration
\begin{equation}\label{hn}
0\,=\, E^{0}\,\subsetneq\, E^{1}\,\subsetneq\,\dots\,\subsetneq E^{m}\,=\,E,
\end{equation}
where each $E^{i}$ is a sub-bundle of $E$ such that the quotient $(E^{i+1}/E^{i})_{\ast}$, $0\, \leq\, i\,
\leq\, m-1$, equipped the induced parabolic structure is parabolic semistable satisfying the condition 
\begin{equation}\label{par}
\textrm{par-}\mu((E^{i}/E^{i-1})_{\ast})\ > \ \textrm{par-}\mu((E^{i+1}/E^{i})_{\ast})
\end{equation}
for all $1\,\leq\, i\, \leq\, m-1$.
The minimal parabolic slope and the minimal parabolic degree of $\e$ are defined as:
\begin{equation*}
\textrm{par-}\mu_{\textrm{min}}(E_{\ast})\,:=\, \textrm{par-}\mu((E^{m}/E^{m-1})_{\ast}) \quad \textrm{and} \quad \textrm{par-}d_{\textrm{min}}(\e)\,:=\, \textrm{par-deg}((E^{m}/E^{m-1})_{\ast}).
\end{equation*}

\begin{lemma}\label{min}
Let $\e$ be a parabolic vector bundle on $X$ with the Harder-Narasimhan filtration as in \eqref{hn}. Let
$Q_{\ast}$ be any quotient parabolic vector bundle of $\e$. Then $\emph{par-}\mu(Q_{\ast})\,\geq\,$ \emph{par-}${\mu}_{\emph{min}}(\e)$.
\end{lemma}

\begin{proof}
The first step will be to show that it suffices to prove the lemma under the assumption that 
the quotient $Q_{\ast}$ is parabolic semistable.

If $Q_{\ast}$ is not parabolic semistable, then it has a parabolic semistable quotient $\widetilde{Q}_{\ast}$ such that
\begin{equation}\label{par1}
\textrm{par-}\mu(Q_{\ast})\ >\ \textrm{par-}\mu(\widetilde{Q}_{\ast}).
\end{equation}
Thus, if we prove the lemma for semistable quotients, i.e.,
\begin{equation}\label{par2}
\textrm{par-}\mu(\widetilde{Q}_{\ast})\ \geq\ \textrm{par-}\mu_{\textrm{min}}(\e),
\end{equation}
then from the combination of \eqref{par1} and \eqref{par2} it follows that $\textrm{par-}\mu(Q_{\ast})\, >\, 
\textrm{par-}\mu_{\textrm{min}}(\e)$. Hence it is enough to consider the case where $Q_{\ast}$ is parabolic semistable.

For any $1\, \leq\, i\, \leq\, m$, we have the composition of maps $E^i_{\ast}\, \hookrightarrow\, E_{\ast}
\, \longrightarrow\, Q_{\ast}$. Define
\begin{equation*}
i_{0}\ :=\ \textrm{min}\{i\,\,\big\vert\,\,\, \textrm{the above map}\,\,E^{i}_{\ast}\,\longrightarrow\,\, Q_{\ast}\, \textrm{is nonzero} \}.
\end{equation*}
Note that $m\, \geq\, i_0\, \geq\, 1$. Hence we have a nonzero morphism 
\begin{equation*}
\varphi\ :\ E^{i_{0}}_{\ast}\ \longrightarrow\ Q_{\ast}
\end{equation*}
such that $\varphi\big\vert_{E^{i_0-1}_{\ast}}\,=\, 0$. Consequently, $\varphi$ descends to a nonzero homomorphism 
\begin{equation}\label{a1}
\widetilde{\varphi}\ :\ E^{i_{0}}_{\ast}/E^{i_{0}-1}_{\ast}\ \longrightarrow\ Q_{\ast}.
\end{equation}

Since both $E^{i_{0}}_{\ast}/E^{i_{0}-1}_{\ast}$ and $Q_{\ast}$ are parabolic semistable, the existence of
the nonzero homomorphism $\widetilde{\varphi}$ in \eqref{a1} ensures that 
\begin{equation*}
\textrm{par-}\mu(E^{i_{0}}_{\ast}/E^{i_{0}-1}_{\ast})\ \leq\ \textrm{par-}\mu(Q_{\ast}).
\end{equation*}
This implies that $\text{par-}\mu(Q_{\ast})\,\geq\, \text{par-}{\mu}_{\rm min}(\e)$ because
$\textrm{par-}\mu(E^{i_{0}}_{\ast}/E^{i_{0}-1}_{\ast})\, \geq\, \text{par-}{\mu}_{\rm min}(\e)$.
\end{proof}

\begin{lemma}\label{cri1}
Let $\e$ be a parabolic vector bundle on $X$ with Harder–Narasimhan filtration as in \eqref{hn}.
Then the following two statements are equivalent:
\begin{itemize}
\item[$(1)$] The parabolic degree {\rm par-deg}$(\e)$ is strictly positive, and the parabolic degree of
any quotient vector bundle of $E$ with the induced parabolic structure is also strictly positive.

\item[$(2)$] ${\rm par\mbox{-}}\mu_{\emph{min}}(E_*)\ >\ 0$. 
\end{itemize}
\end{lemma}

\begin{proof}
$(1) \implies (2)$.

By definition par-$\mu_{\textrm{min}}(\e)\,=\, \textrm{par-}\mu(E/E^{m-1})_{\ast}$, which is a quotient of
$E$ with the induced parabolic structure. Hence, by hypothesis $(1)$ we have par-$\mu_{\textrm{min}}(\e)\,>\,0$. 
 
 $(2)\implies (1)$.

Let $Q_{\ast}$ be any quotient of $E$ equipped with induced parabolic structure. By Lemma \ref{min}, 
\begin{equation}\label{e1}
\textrm{par-}\mu(Q_{\ast})\ \geq\ \textrm{par-}\mu_{\textrm{min}}(\e).
\end{equation}
By (2), par-$\mu_{\textrm{min}}(\e)\,>\,0$, and hence from \eqref{e1} it
follows that par-$\mu(Q_{\ast})\,>\, 0$. By definition 
\begin{equation*}
\textrm{par-}\mu(Q_{\ast})\ :=\ \textrm{par-deg}(Q_{\ast})/\textrm{rank}(Q_{\ast}),
\end{equation*}
this implies that par-deg$(Q_{\ast})\,>\,0$.

Since par-$\mu(\e)\,\geq \,\textrm{par-}\mu_{\textrm{min}}(\e)$,
it follows again that par-deg$(\e)\,>\,0$. This proves $(1)$. 
\end{proof}

We now give an equivalent condition for a parabolic vector bundle to be parabolic ample.

\begin{theorem}\label{cri2}
Let $\e$ be a parabolic vector bundle on $X$. Then the following three statements are equivalent: 
\begin{enumerate}
\item [$(1)$] For every vector bundle $V$ on $X$, there exists an integer $n_{0}$ depending on $V$ such that for
all $n\,\geq\,n_{0}$, the vector bundle $S^{n}(\e)_{0}\,\otimes_{\mathcal{O}_{X}}V$ is
generated by its global sections.

\item [$(2)$] The parabolic degree {\rm par-deg}$(\e)$ is strictly positive, and the parabolic degree of every 
quotient vector bundle of $E$, equipped with the induced parabolic structure, is also strictly positive.

\item [$(3)$] The inequality \emph{par-}$\mu_{\emph{min}}(\e)\,>\,0$ holds. 
\end{enumerate}
\end{theorem}

\begin{proof}
The equivalence $(1) \iff (2)$ follows by \cite[Theorem 3.1]{Bi2}. The equivalence $(2)\iff (3)$ follows from
Lemma \ref{cri1}.
\end{proof}

\begin{definition}\label{ample}
A parabolic vector bundle $\e$ is said to be parabolic ample if it satisfies any one of the conditions in
Theorem~\ref{cri2}.
\end{definition}

Next, we define the parabolic analogue of nef vector bundles; see \cite{BS} for more details.

\begin{definition}\label{den}
A parabolic vector bundle $\e$ is called \textrm{parabolic nef} if there is an ample line bundle $L$ over $X$
such that $S^{k}(\e)\,\otimes\,L$ is parabolic ample for every $k$, where $S^{k}(\e)$ denote the $k$-fold
parabolic symmetric tensor power of the parabolic vector bundle $\e$. Here, $L$ is equipped with the trivial
parabolic structure, meaning it has no nonzero parabolic weights.
\end{definition}

We now give a criterion for parabolic vector bundle to be parabolic nef.

\begin{proposition}\label{p}
A parabolic vector bundle $\e$ on a smooth complex projective curve $X$ is parabolic nef if and only if, inequality 
\begin{equation*}
{\rm par}\textrm{-}\mu_\emph{min}(E_{\ast})\ \geq\ 0
\end{equation*}
holds.
\end{proposition}

\begin{proof}
Assume, we have
\begin{equation}\label{degree}
\textrm{par-}\mu_{\textrm{min}}(\e)\ \geq\ 0. 
\end{equation}
To prove that $\e$ is parabolic nef, take an ample line bundle $L$ on $X$ with trivial parabolic weights. 
Our goal is to prove that for every positive integer $k$, the parabolic vector bundle $S^{k}(\e)\,\otimes\,L$
is parabolic ample.

From the definition of parabolic tensor product it is easy to see that for any two parabolic vector
bundles $V_{\ast}$ and $W_{\ast}$ on the curve $X$,
\begin{equation}\label{tensor}
\textrm{par-deg}(V_{\ast}\,\otimes\,W_{\ast})\,=\, \textrm{rank}(V_{\ast}).\textrm{par-deg}\,W_{\ast}\,+\,
\textrm{rank}(W_{\ast}).\textrm{par-deg}\,V_{\ast}.
\end{equation}
Take $L$ as in Definition \ref{den}. Since $L$ is a parabolic ample line bundle with trivial parabolic weights,
\begin{equation}\label{j1}
\textrm{par-deg}(L)\ =\ \textrm{deg}(L)\ >\ 0.
\end{equation}
Using this and \eqref{j1} in \eqref{tensor} we conclude that
\begin{equation*}
\textrm{par-deg}(S^{k}(\e)\,\otimes\,L)\ \geq\ 0.
\end{equation*}

Next, we will show that every parabolic quotient of $S^{k}(\e)\,\otimes\, L$ has strictly positive parabolic
degree. Since par-$\mu_{\textrm{min}}(\e)\,\geq\,0$, it follows that
\begin{equation}\label{j2}
\textrm{par-}d_{\textrm{min}}S^{k}(\e)
\ =\ k\cdot\textrm{par-}d_{\textrm{min}}(\e)\ \geq\ 0 .
\end{equation}

A consequence of \eqref{j2} is that any quotient parabolic vector bundle of $S^{k}(\e)$
is of nonnegative parabolic degree (see Lemma \ref{min}). Hence any parabolic quotient of 
$S^{k}(\e)\,\otimes\,L$ is of strictly positive degree. Therefore Theorem \ref{cri2} says that 
$S^{k}(\e)\,\otimes\,L$ is parabolic ample. Thus $\e$ is parabolic nef.

To prove the converse, assume that $\e$ is parabolic nef. We will show that
\begin{equation*}
\textrm{par-}\mu_{\textrm{min}}(\e)\ \geq\ 0.
\end{equation*}
Since 
$\e$ is parabolic nef, there exists an ample line bundle 
$L$ on $X$ such that, for any $k\, \geq\, 1$,
\begin{equation*}
\textrm{par-}\mu_{\textrm{min}}(S^{k}(\e)\,\otimes\,L)\ =\
\textrm{par-}\mu_{\textrm{min}}(S^{k}(\e))\,+\, \textrm{par-}\mu_{\textrm{min}}(L)
\end{equation*}
\begin{equation*}
\hspace{3cm} =\ k.\textrm{par-}\mu_{\textrm{min}}(\e)\,+\,\textrm{deg}(L).
\end{equation*}
Since $\e$ is parabolic nef, we have
\begin{equation}\label{positive}
\textrm{par-}\mu_{\textrm{min}}(S^{k}(\e)\,\otimes\,L)\,>\,0 
\end{equation}
for all $k$. If $\textrm{par-}\mu_{\textrm{min}}(\e)\,<\,0$, then, as $L$ is ample and hence deg$\,L\,>\,0$,
we can choose $k$ large enough so that 
$\textrm{par-}\mu_{\textrm{min}}(S^{k}(\e)\,\otimes\,L)\,<\,0$. But this
contradicts \eqref{positive}. Hence $\textrm{par-}\mu_{\textrm{min}}(\e)\,\geq\,0$.
\end{proof}

\section{Pullback of parabolic ample and parabolic nef vector bundles}

Let $\theta\,:\, Z_{1}\,\longrightarrow\, Z_{2}$ be a non-constant morphism between smooth complex
projective curves, and let $F_{\ast}$ be a parabolic vector bundle on $Z_{2}$. Then the pulled back
parabolic vector bundle $\theta^{\ast}F_{\ast}$ on $Z_{1}$ carries a natural parabolic structure;
see \cite[Section $3$]{AB}. The parabolic divisor for $\theta^{\ast}F_{\ast}$ is the reduced inverse image
\begin{equation*}
\theta^{-1}(D)_{\textrm{red}},
\end{equation*}
where $D$ is the parabolic divisor for $F_{\ast}$.

\begin{theorem}\label{pull}
Let $X$ and $Y$ be smooth irreducible projective curves and $f\,:\, Y \,\longrightarrow\, X$ a
non-constant morphism. Then a parabolic vector bundle $\e$ on $X$ is parabolic
ample (respectively, parabolic nef) if and only if
$f^{\ast}\e$ on $Y$ is parabolic ample (respectively, parabolic nef).
\end{theorem}

\begin{proof} It is known (see \cite[Lemma 3.5(2)]{AB}) that the pullback of the Harder-Narasimhan filtration of $\e$
\begin{equation}
0\,=\, f^{\ast}E^{0}\,\subset\, f^{\ast}E^{1}\,\subset\,\dots\,\subset f^{\ast}E^{m}\,=\, f^{\ast}E
\end{equation} 
coincides with the Harder-Narasimhan filtration of $f^{\ast}\e$.
Therefore, the minimal parabolic slope of $f^{\ast}\e$ is
\begin{equation}\label{rel}
\textrm{par-}\mu_{\textrm{min}}(f^{\ast}\e)\ =\ \textrm{par-}\mu(f^{\ast}(E/E^{m-1})_{\ast})\ =\
\textrm{degree}(f)\cdot \textrm{par-}\mu_{\textrm{min}}(\e).
\end{equation}
The last equality follows due to the fact that 
\begin{equation*}
\textrm{par-deg}(f^{\ast}(E/E^{m-1})_{\ast})\,=\, \textrm{degree}(f)\cdot \textrm{par-deg}(E/E^{m-1})_{\ast}.
\end{equation*}
By \eqref{rel} it is clear that 
\begin{equation*}
\textrm{par-}\mu_{\textrm{min}}(f^{\ast}\e)\,>\,0 \, \iff \textrm{par-}\mu_{\textrm{min}}(\e)\,>\,0,
\end{equation*}
$(\textrm{respectively, par-}\mu_{\textrm{min}}(f^{\ast}\e)\,\geq\,0 \iff \textrm{par-}\mu_{\textrm{min}}(\e)\,\geq\,0)$. Hence, by Theorem \ref{cri2} (respectively, Proposition \ref{p}) $\e$ is parabolic ample (respectively, parabolic nef) if and only if $f^{\ast}\e$ is parabolic ample (respectively, parabolic nef). 
\end{proof}

\begin{definition}
A parabolic vector bundle $\e$ on a complex projective variety is said to be parabolic anti-ample
(respectively, parabolic anti-nef) if the dual parabolic vector bundle $\e^{\vee}$ is parabolic ample
(respectively, parabolic nef).
\end{definition}

\begin{corollary}
Let $X$ and $Y$ be smooth irreducible projective curves and $f\,:\, Y \,\longrightarrow\, X$ a
non-constant morphism. Then a parabolic vector bundle $\e$ on $X$ is parabolic anti-ample (respectively, parabolic anti-nef) if and only if
$f^{\ast}\e$ on $Y$ is parabolic anti-ample (respectively, parabolic anti-nef).
\end{corollary}

\begin{proof}
The parabolic vector bundle $\e$ is parabolic anti-ample (respectively, anti-nef) if and only if the parabolic dual 
$E_{\ast}^{\vee}$ is parabolic ample (respectively, parabolic nef) (by definition). By Theorem \ref{pull}, the 
parabolic dual $E_{\ast}^{\vee}$ is parabolic ample (respectively, parabolic nef) if and only if 
$f^{\ast}(\e^{\vee})$ is parabolic ample (respectively, parabolic nef). On the other hand, by \cite[Remark 
3.4]{AB}, the operations of pullback and dualization commute, so $f^{\ast}(\e^{\vee})$ is the same as 
$(f^{\ast}\e)^{\vee}$. Therefore, $f^{\ast}(\e^{\vee})$ is parabolic ample (respectively, parabolic nef) if 
and only if $(f^{\ast}\e)^{\vee}$ is parabolic ample (respectively, parabolic nef) . But $(f^{\ast}\e)^{\vee}$ 
is parabolic ample (respectively, parabolic nef) if and only if $f^{\ast}\e$ is anti-ample (respectively, 
anti-nef).
\end{proof}

\section{Direct image of parabolic ample and parabolic nef vector bundles }

Let $\phi\,:\, Z_{1}\,\longrightarrow\, Z_{2}$ be a non-constant morphism between smooth complex
projective curves, and let $\e$ be a parabolic vector bundle on $Z_{1}$ with parabolic divisor $D$. Let
\begin{equation*}
R \ \subset \ X
\end{equation*}
be the ramification locus of $\phi$. Define a subset 
\begin{equation*}
\Delta\,=\, \phi\,(R\,\cup\, D)\,\subset\,Z_{2}.
\end{equation*}
A natural parabolic
structure on the direct image $\phi_{\ast}\e$, with parabolic divisor is $\Delta$, was constructed in \cite[Section 4]{AB}. 

\begin{theorem}\label{push}
Let $X$ and $Y$ be an irreducible smooth complex projective curves, and let $f \,:\, Y \,\longrightarrow\, X$
be a non-constant morphism. Then a parabolic vector bundle $\V$ on $Y$ is parabolic ample (respectively, parabolic nef) if and only if its pushforward $f_{\ast}\V$ on $X$ is parabolic ample (respectively, parabolic nef).
\end{theorem}

\begin{proof}
Consider the Galois closure
\begin{equation}\label{h}
h\ :\ Z\ \longrightarrow\ X
\end{equation}
of the above map $f$. Let
\begin{equation}\label{Aut}
\Gamma\ :=\ \text{Gal}(h)\ =\ \text{Aut}(Z/X)
\end{equation}
be the Galois group of this ramified covering $h$. There is a natural map
\begin{equation}\label{g}
g\ :\ Z\ \longrightarrow\ Y
\end{equation}
such that $h$ factors as
\begin{equation*}
f\,\circ\, g\ =\ h.
\end{equation*}

First assume that $\V$ is parabolic ample (respectively, parabolic nef). To prove that
$f_{\ast}\V$ is parabolic ample (respectively, parabolic nef), it is enough to show that
$h^{\ast}(f_{\ast}\V)$ is parabolic ample (respectively, parabolic nef) (see Theorem \ref{pull}).

Since $h$ is a (ramified) Galois morphism, by \cite[Proposition 4.2(2)]{AB}, we have
\begin{equation}\label{sum}
h^{\ast}h_{\ast} g^{\ast}\V \ =\ \bigoplus_{\gamma\,\in\, \Gamma} \gamma^{\ast}g^{\ast}V_{\ast},
\end{equation}
where $\Gamma$ is the Galois group in \eqref{Aut}. As $f\,\circ\,g\,=\,h$, the
parabolic vector bundle $h^{\ast}(f_{\ast}V_{\ast})$ is a parabolic
sub-bundle of the parabolic vector bundle $h^{\ast}h_{\ast}g^{\ast}V_{\ast}$. Consequently, by \eqref{sum},
we have 
\begin{equation}\label{sub}
h^{\ast}(f_{\ast}V_{\ast})\ \subset\ \bigoplus_{\gamma\,\in\, \Gamma} \gamma^{\ast}g^{\ast}V_{\ast}
\end{equation}
as a parabolic sub-bundle. Now we will explicitly describe the parabolic sub-bundle in \eqref{sub}.

Let
\begin{equation*}
G\ :=\ \textrm{Gal}(g)\,=\, \textrm{Aut}(Z/Y) 
\end{equation*}
be the Galois group of the map $g$ in \eqref{g}. So $G$ is a subgroup of the group $\Gamma$
defined in \eqref{Aut}, and $Y\,=\, Z/G$. Note that $G$ is a normal subgroup of $\Gamma$ if and only if
the map $f$ is (ramified) Galois. There is a natural action of $G$ on $g^{\ast}V_{\ast}$ over the
Galois action of $G$ on $Z$. Take any $w\,\in\, (g^{\ast}V_{\ast})_{z}$, $z\,\in\, Z$, and any $t\,\in\,G$. The point of $(g^{\ast}V_{\ast})_{(t(z))}$ to which $w$ is taken by the action of $t$ will be denoted
by $t\cdot w$.

The action of $G$ on $g^{\ast}V_{\ast}$ (over the action of $G$ on $Z$) produces
an action of $G$ on 
\begin{equation}\label{V}
\mathcal{V}_{\ast} \,:=\, \bigoplus_{\gamma\in \Gamma}\gamma^{\ast}g^{\ast}V_{\ast}
\end{equation}
over the trivial action of $G$ on $Z$. We will explicitly describe this action of $G$ on the parabolic
vector bundle $\mathcal{V}_{\ast}$ in \eqref{V}.

Take any point $z\,\in\, Z$. The fiber of the parabolic vector bundle $\mathcal{V}_{\ast}$ over $z$ is 
\begin{equation*}
(\mathcal{V}_{\ast}){_{z}}\ =\ \bigoplus_{\gamma\in\Gamma} (g^{\ast}V_{\ast})_{\gamma(z)}.
\end{equation*}
Take any element $\bigoplus_{\gamma\in \Gamma}w_{\gamma}\,\in\, \bigoplus_{\gamma\in \Gamma}(g^{\ast}V_{\ast})_{\gamma(z)}$,
where $w_{\gamma}\,\in\, (g^{\ast}V_{\ast})_{\gamma(z)}\,=\, (V_{\ast}){_{g(\gamma(z))}}$. The action of any $t\,\in\, G$ sends 
$\bigoplus_{\gamma\in \Gamma}w_{\gamma}$ to $\bigoplus_{\gamma\in \gamma} t\cdot w_{t^{-1}\gamma}$. The parabolic
sub-bundle in \eqref{sub} has the following description:
\begin{equation}\label{invariant}
h^{\ast}(f_{\ast}V_{\ast})\,=\, \mathcal{V}_{\ast}^{G}\,\subset\, \mathcal{V}_{\ast}\,=\, \bigoplus_{\gamma\in\Gamma}\gamma^{\ast}g^{\ast}V_{\ast},
\end{equation}
where $\mathcal{V}_{\ast}^{G}$ is the invariant parabolic sub-bundle for the above action $G$ on $\mathcal{V}_{\ast}$.

Now we will give a more detailed description of $h^{\ast}(f_{\ast}V_{\ast})$. Fix a subset $\widetilde{\Gamma}
\,\subset\, \Gamma$ such that the following composition of maps is a bijection:
\begin{equation}\label{subset}
\widetilde{\Gamma}\ \hookrightarrow\ \Gamma\ \longrightarrow\ \Gamma/G,
\end{equation}
where $\Gamma\,\longrightarrow\,\Gamma/G$ is the quotient map to the right quotient space $\Gamma/G$
(as mentioned before, in general $G$ is not a normal subgroup of $\Gamma$), and 
\begin{equation*}
\widetilde{\Gamma}\,\cap\, G\,=\, \{e\}\,\,\textrm{(the identity element of}\, \Gamma).
\end{equation*}
From \eqref{invariant} it follows that the parabolic sub-bundle $h^{\ast}(f_{\ast}V_{\ast})\,\subset\,
\bigoplus_{\gamma\in\Gamma}\gamma^{\ast}g^{\ast}V_{\ast}$ is isomorphic to the direct sum
$\bigoplus_{\gamma\in\widetilde{\Gamma}}\gamma^{\ast}g^{\ast}V_{\ast}$, where $\widetilde{\Gamma}$ is the subset \eqref{subset}. In fact, we have an isomorphism 
\begin{equation*}
h^{\ast}(f_{\ast}V_{\ast})\,\longrightarrow\, \bigoplus_{\gamma\in \widetilde{\Gamma}} \gamma^{\ast}g^{\ast}V_{\ast}
\end{equation*}
which is the composition
of the inclusion map $h^{\ast}(f_{\ast}V_{\ast})\,\hookrightarrow\, \bigoplus_{\gamma\in\Gamma}\gamma^{\ast}g^{\ast}V_{\ast}$ (see \eqref{invariant}) with the natural projection 
\begin{equation*}
\bigoplus_{\gamma\in \Gamma} \gamma^{\ast}g^{\ast} V_{\ast} \,\longrightarrow\,
\bigoplus_{\gamma\in\widetilde{\Gamma}}\gamma^{\ast}g^{\ast}V_{\ast}
\end{equation*}
given by the inclusion map $\widetilde{\Gamma}\,\hookrightarrow\,\Gamma$.

Thus, 
\begin{equation}\label{tgamma}
h^{\ast}(f_{\ast}V_{\ast})\ \cong\ \bigoplus_{\gamma\in \widetilde{\Gamma}} \gamma^{\ast}g^{\ast}V_{\ast}.
\end{equation}
Since $V_{\ast}$ is parabolic ample (respectively, parabolic nef), Theorem \ref{pull} says that $\gamma^{\ast}g^{\ast}V_{\ast}$ is parabolic
ample (respectively, parabolic nef) for each ${\gamma}\,\in\, \Gamma$. So $\bigoplus_{{\gamma}\in \widetilde{\Gamma}} \gamma^{\ast}g^{\ast} V_{\ast}$ is
parabolic ample (respectively, parabolic nef). A direct summand of parabolic ample vector bundle (respectively, parabolic nef vector bundle) is again parabolic ample (respectively, parabolic nef), and hence from
\eqref{tgamma} it follows that $h^{\ast}(f_{\ast}V_{\ast})$ is parabolic ample (respectively, parabolic nef). Again by Theorem
\ref{pull} the parabolic vector bundle $f_{\ast}V_{\ast}$ is parabolic ample (respectively, parabolic nef).

To prove the converse, assume that $f_{\ast}V_{\ast}$ is parabolic ample (respectively, parabolic nef).
We will show that $V_{\ast}$ is parabolic ample (respectively, parabolic nef).

The parabolic vector bundle $h^{\ast}(f_{\ast}V_{\ast})$ is parabolic ample (respectively, parabolic nef) 
because $f_{\ast}V_{\ast}$ is parabolic ample (respectively, parabolic nef) (see Theorem \ref{pull}). So from 
\eqref{tgamma} it follows that the direct sum $\bigoplus_{\gamma\in \widetilde{\Gamma}} \gamma^{\ast} 
g^{\ast}V_{\ast}$ is parabolic ample (respectively, parabolic nef). Consequently, 
$\gamma^{\ast}g^{\ast}V_{\ast}\, \subset\, \bigoplus_{\gamma\in \widetilde{\Gamma}} \gamma^{\ast}g^{\ast} 
V_{\ast}$ is parabolic ample (respectively, parabolic nef) for every $\gamma\,\in\, \widetilde{\Gamma}$. Now 
Theorem \ref{pull} says that $V_{\ast}$ is parabolic ample (respectively, parabolic nef).
\end{proof}

\begin{remark}
The direct image of a nef vector bundle need not be nef. For example,
for a degree two map $f\,:\, {\mathbb P}^1\, \longrightarrow\, {\mathbb P}^1$, we have
$f_* {\mathcal O}_{{\mathbb P}^1} \,=\, {\mathcal O}_{{\mathbb P}^1}\oplus {\mathcal O}_{{\mathbb P}^1}(-1)$.
\end{remark}

\begin{lemma}
The parabolic direct image commutes with parabolic dual; that is,
\begin{equation*}
f_{\ast}(V_{\ast}^{\vee})\ =\ (f_{\ast}V_{\ast})^{\vee}.
\end{equation*}
\end{lemma}

\begin{proof}
As in \eqref{h}, let $h\ :\ Z\ \longrightarrow\ X$ denote the Galois closure of $f$.
It suffices to show that
\begin{equation}\label{s1}
h^\ast (f_{\ast}(V_{\ast}^{\vee}))\ =\ h^\ast ((f_{\ast}V_{\ast})^{\vee})
\end{equation}
as $\Gamma$--equivariant vector bundles, where $\Gamma$ is the Galois group of $h$ (see \eqref{Aut}).
Since $(h^{\ast}f_{\ast}V_{\ast})^{\vee}\, \cong\, h^\ast ((f_{\ast}V_{\ast})^{\vee})$
(see \cite[Remark 3.4]{AB}), we know that \eqref{s1} is equivalent to the following:
\begin{equation}\label{s2}
h^\ast (f_{\ast}(V_{\ast}^{\vee}))\ =\ (h^{\ast}f_{\ast}V_{\ast})^{\vee}
\end{equation}
as $\Gamma$--equivariant vector bundles.

Recall the description $h^{\ast}(f_{\ast}V_{\ast})\,\cong\, \bigoplus_{\gamma\in \widetilde{\Gamma}}
\gamma^{\ast}g^{\ast}V_{\ast}$ in \eqref{tgamma}. Taking the parabolic dual of both sides:
\begin{equation*}
(h^{\ast}f_{\ast}V_{\ast})^{\vee}\ \cong\ \left(\bigoplus_{\gamma\in \widetilde{\Gamma}}\gamma^{\ast}
g^{\ast}V_{\ast}\right)^{\vee} \ \cong\
\bigoplus_{\gamma\in\widetilde{\Gamma}} \left( \gamma^{\ast}g^{\ast}V_{\ast} \right)^{\vee}.
\end{equation*}
As taking parabolic dual commutes with the pullback operation (see \cite[Remark 3.4]{AB}), we have:
\begin{equation}\label{dual}
(h^{\ast}f_{\ast}V_{\ast})^{\vee}\ \cong\ 
\bigoplus_{\gamma\in\widetilde{\Gamma}} \left( \gamma^{\ast}g^{\ast}V_{\ast} \right)^{\vee}\ \cong\
\bigoplus_{\gamma\in\widetilde{\Gamma}} \gamma^{\ast}g^{\ast}\left(V_{\ast} ^{\vee} \right).
\end{equation}
Substituting $V_{\ast}^\vee$ in place of $V_{\ast}$ in \eqref{tgamma},
$$
h^{\ast}(f_{\ast}(V^\vee_{\ast}))\ \cong\ \bigoplus_{\gamma\in \widetilde{\Gamma}}
\gamma^{\ast}g^{\ast}(V_{\ast}^\vee) .
$$
Combining this with \eqref{dual} we conclude that \eqref{s2} holds. This completes the proof.
\end{proof}

\begin{corollary}
Let $X$ and $Y$ be smooth irreducible projective curves, and let $f \,:\, Y \,\longrightarrow\, X$ be a 
non-constant morphism. A parabolic vector bundle $\V$ on $Y$ is parabolic anti-ample (respectively, parabolic 
anti-nef) if and only if its pushforward $f_{\ast}\V$ on $X$ is parabolic anti-ample (respectively, parabolic 
anti-nef).
\end{corollary}

\begin{proof}
Assume that $\V$ is parabolic anti-ample (respectively, parabolic anti-nef). By definition, the dual $V_{\ast}^{\vee}$ is parabolic ample (respectively, parabolic nef). From
Theorem \ref{push} we know that $f_{\ast} (\V^{\vee})$ is parabolic ample (respectively, parabolic nef). So
$f_{\ast} (\V^{\vee})\,=\, (f_{\ast} \V)^{\vee}$ is parabolic ample (respectively, parabolic nef). In other words,
$f_{\ast}\V$ is parabolic anti-ample (respectively, parabolic anti-nef).

To prove the converse, assume that $f_{\ast}\V$ is parabolic anti-ample (respectively, parabolic anti-nef). So
$f_{\ast} (\V^{\vee})\,=\, (f_{\ast} \V)^{\vee}$ is parabolic ample (respectively, parabolic nef). Theorem \ref{push} says that
$\V^{\vee}$ is parabolic ample (respectively, parabolic nef). Hence $\V$ is parabolic anti-ample (respectively, parabolic anti-nef).
\end{proof}

\section*{Acknowledgements}

The second named author is partially supported by a J. C. Bose Fellowship (JBR/2023/000003).

\section*{Declarations}

On behalf of all authors, the corresponding author states that there is no conﬂict of interest. No data
were used or generated.

\end{document}